\documentclass[11pt]{article}

\usepackage[margin=1.25in]{geometry}

\usepackage{amssymb,amsmath}
\usepackage{amsthm}
\usepackage{hyperref}
\usepackage{mathrsfs}
\usepackage{textcomp}
\hypersetup{
    colorlinks,%
    citecolor=black,%
    filecolor=black,%
    linkcolor=black,%
    urlcolor=black
}

\usepackage{verbatim}

\usepackage{sectsty}

\makeatletter

\newdimen\bibspace
\renewenvironment{thebibliography}[1]{%
 \section*{\refname 
       \@mkboth{\MakeUppercase\refname}{\MakeUppercase\refname}}%
     \list{\@biblabel{\@arabic\c@enumiv}}%
          {\settowidth\labelwidth{\@biblabel{#1}}%
           \leftmargin\labelwidth
           \advance\leftmargin\labelsep
           \itemsep\bibspace
           \parsep\z@skip     %
           \@openbib@code
           \usecounter{enumiv}%
           \let\p@enumiv\@empty
           \renewcommand\theenumiv{\@arabic\c@enumiv}}%
     \sloppy\clubpenalty4000\widowpenalty4000%
     \sfcode`\.\@m}
    {\def\@noitemerr
      {\@latex@warning{Empty `thebibliography' environment}}%
     \endlist}

\makeatother

\makeatletter

\newtheorem{thm}{Theorem}[section]
\newtheorem{lem}[thm]{Lemma}
\newtheorem{prop}[thm]{Proposition}

\newtheorem{rem}[thm]{Remark}

\def\XXint#1#2#3{{\setbox0=\hbox{$#1{#2#3}{\int}$}
  \vcenter{\hbox{$#2#3$}}\kern-.5\wd0}}

\newcommand{\al}{\alpha}                \newcommand{\lda}{\lambda}
                \newcommand{\pa}{\partial}
\newcommand{\va}{\varepsilon}           \newcommand{\ud}{\mathrm{d}}
\newcommand{\be}{\begin{equation}}      \newcommand{\ee}{\end{equation}}

\newcommand{\R}{\mathbb{R}}

\newcommand{\abs}[1]{\lvert#1\rvert}

\begin{document}

\title{\textbf{Singular extinction profiles of solutions to some fast diffusion equations}
\bigskip}

\author{\medskip  Tianling Jin\footnote{T. Jin is partially supported by Hong Kong RGC grants GRF 16302217 and GRF 16306320.}, \  \
Jingang Xiong\footnote{J. Xiong is partially supported by NSFC 11922104 and 11631002.}}

\date{\today}

\maketitle

\begin{abstract}
We study extinction profiles of solutions to fast diffusion equations with some initial data in the Marcinkiewicz space. The extinction profiles will be the singular solutions of their stationary equations. 
\end{abstract}


\section{Introduction} 

In this paper, we study the non-negative solution of the fast diffusion equation
\begin{equation}\label{eq:fde0}
\begin{split}
u_t&=\Delta u^m\quad\mbox{in }\R^n\times(0,\infty),\\
u(\cdot,0)&=u_0\ge 0,
\end{split}
\end{equation}
where $0<m<1$, $n\ge 3$, and $u_0$ is assumed to be non-negative and locally integrable. It is well known that this problem is well-posed for such $u_0$; see  Herrero-Pierre \cite{HP}. It is also known that when $m>\frac{n-2}{n}$, the solution is positive and smooth at all positive times. However, when $m<\frac{n-2}{n}$, locally integrable initial data may not produce locally bounded solutions, and the solution may be extinct after a finite time, i.e., there exists $T^*>0$ such that $u(x,t)>0$ for all $0<t<T^*$ and $u(x,t)\equiv 0$ for all $t\ge T^*$. An explicit example for such solutions to \eqref{eq:fde0} is
\begin{equation}\label{eq:example}
U(x,t;T^*)=\left(\frac{2m(n-2-nm)}{1-m}\right)^{\frac{1}{1-m}}\left(\frac{T^*-t}{|x|^2}\right)^{\frac{1}{1-m}}
\end{equation}
with arbitrary $T^*>0$.

There have been many interests in analyzing the exact behavior of the solutions near the vanishing time. In \cite{King}, King first formally provided vanishing profiles in the radially symmetric case, and later Galaktionov-Peletier \cite{GP} proved rigorously some of the conjectures raised by King. When $m=\frac{n-2}{n+2}$, del Pino-S\'aez \cite{del} obtained the extinction profile for any fast-decay initial condition without any radial symmetry assumption. This case is of special interest, since it is the so-called Yamabe flow, and the equation is conformally invariant. Blanchet-Bonforte-Dolbeault-Grillo-V\'azquez \cite{BB+}, Bonforte-Grillo-V\'azquez \cite{BGV} and Daskalopoulos-Sesum \cite{DS} showed that if the initial value $u_0$ is bounded from below and above by two Barenblatt solutions (which are self-similar solution to \eqref{eq:fde0} and decay at the rate of $|x|^{-\frac{2}{1-m}}$ at infinity) and approaching to a Barenblatt solution near the infinity in some sense, then the asymptotic behavior of $u$ at the vanishing time is given by a Barenblatt solution. Bonforte-Simonov \cite{BS20} characterized the maximal set of initial data that produces solutions which are pointwisely trapped between two Barenblatt solutions, and uniformly converge in relative error. Further results on the convergence rates can be found in Fila-V\'azquez-Winkler-Yanagida \cite{FVWY}, and convergence  to other  similar solutions can be found in  Daskalopoulos-King-Sesum \cite{DKS}. See also Fila-Winkler \cite{FW} for a result when the initial data are close to a radial stationary (after rescaling the time) solution of  \eqref{eq:fde0}. Asymptotic behavior of singular solution of some fast diffusion equation in the punctured Euclidean space was recently studied in Hui-Park \cite{HuiPark}.

Here, we would like to study the extinction profiles of solutions $u$ to \eqref{eq:fde0} when the initial data are in the Marcinkiewicz space. Recall that the Marcinkiewicz space $\mathcal{M}^q(\R^n)$ is defined as the set of functions $f\in L^1_{loc}(\R^n)$ such that
\begin{equation}\label{eq:Marcinkiewicz}
\int_K |f(x)|\,\ud x\le C |K|^{\frac{q-1}{q}}
\end{equation}
for all subsets $K$ of finite measure, where $|K|$ denotes the Lebesgue measure of $K$. The minimum $C$ in \eqref{eq:Marcinkiewicz} gives a norm in this space, i.e.,
\[
\|f\|_{\mathcal{M}^q(\R^n)}=\sup\{|K|^{-\frac{q-1}{q}}\int_K |f(x)|\,\ud x: K\subset\R^n, |K|<\infty \}.
\]
One can refer to the appendix of B\'enilan-Brezis-Crandall \cite{BBC} on more properties of Marcinkiewicz spaces. Note that $U(x,0;T)$ defined in \eqref{eq:example} belongs to $\mathcal{M}^{q_*}(\R^n)$, where $$q_*=\frac{n(1-m)}{2}.$$ 

When $0<m<\frac{n-2}{n}$, the Marcinkiewicz space has been characterized as a natural space for regularity and extinction of solutions to \eqref{eq:fde0}, by the results in Chapter 5 of the lecture notes \cite{Vaz} of V\'azquez. If $u_0\in\mathcal{M}^q(\R^n)$ for $q>q_*$, then the solution is uniformly bounded. The function $U$ in \eqref{eq:example} shows that if $u_0\in\mathcal{M}^{q_*}(\R^n)$, then the solution may not be bounded. On the other hand, if $u_0\in\mathcal{M}^{q_*}(\R^n)$, then the solution will be extinct after a finite time, and one has an estimate for the extinction time: $T^*\le C(m,n)\|u_0\|_{\mathcal{M}^{q_*}(\R^n)}^{1-m}$. Conversely, if the solution vanishes after a fine time $T$, then it is necessary that 
\[
|B_R|^{-\frac{q*-1}{q*}}\int_{B_R(x_0)} u_0(x)\,\ud x\le C(m,n) T^{\frac{1}{1-m}},
\]
for all $x_0\in\R^n$ and all $R>0$. Therefore, $\mathcal{M}^{q_*}(\R^n)$ is almost the correct space for extinction of non-negative solutions.

We are interested in  studying the extinction profiles of solutions to \eqref{eq:fde0} when the initial data $u_0$ belongs to $\mathcal{M}^{q_*}(\R^n)$. The most typical function in $\mathcal{M}^{q_*}(\R^n)$ is $|x|^{-\frac{2}{1-m}}$. In this paper, we will suppose that the initial datum $u_0$ is of the form
\begin{equation}\label{eq:initialdata}
u_0(x)=|x|^{-\frac{2}{1-m}}f(x), \mbox{ where } f \mbox{ is positive and bounded away from }0\mbox{ and }\infty.
\end{equation}
Such $u_0$ belongs to $\mathcal{M}^{q_*}(\R^n)$. It decays at the same rate $|x|^{-\frac{2}{1-m}}$ as Barenblatt solutions, and it also blows up at the origin at the rate of $|x|^{-\frac{2}{1-m}}$.   In fact, the function $|x|^{-\frac{2}{1-m}}$ is called a singular Barenblatt solution. When the initial data are perturbations of $|x|^{-\frac{2}{1-m}}$ by the order of $|x|^{-l}$ for some $l> \frac{2}{1-m}+2$, then  convergence of the solutions and the rate of convergence in exterior domains have been studied by Fila-V\'azquez-Winkler \cite{FVW} and Fila-Winkler \cite{FW2}.

\subsection{Cylindrical coordinates}
We will write the equation \eqref{eq:fde0} in cylindrical coordinates for the function $u^m$, whose advantage is to transform the equation \eqref{eq:fde0} with an isolated singularity to a corresponding problem without singularities on a product manifold. This is inspired by the work of Gidas-Spruck \cite{GS} on the study of solutions to Yamabe type equations with an isolated singularity.

Let $p=1/m$, $r=|x|$, $\theta=\frac{x}{|x|}$, $\rho =\ln r$, and
\be \label{eq:cylinercoordinate}
w(\rho,\theta,t)=  r^{\frac{2}{p-1}} u^m(r\theta,t).
\ee
Then
\begin{equation}\label{eq:fdew00}
\begin{split}
\frac{\pa }{\pa t} w^{p}&=\frac{\pa^2}{\pa \rho^2}w+\Delta_{\mathbb{S}^{n-1}} w +a \frac{\pa}{\pa \rho } w+ bw \quad \mbox{on }\Sigma\times (0,\infty),\\
w(\rho,\theta,0)&=f(e^\rho\theta)^{\frac 1p}:=w_0,
\end{split}
\end{equation}
where $\Sigma=(-\infty,\infty)\times \mathbb{S}^{n-1}$, $\Delta_{\mathbb{S}^{n-1}}$ is the  Beltrami-Laplace operator on the standard sphere $\mathbb{S}^{n-1}$, and
\begin{align}\label{eq:bconstant}
a= \frac{n-2}{p-1}\left(p-\frac{n+2}{n-2}\right),\quad
b= \frac{2(n-2)}{(p-1)^2}\left(\frac{n}{n-2} -p\right).
\end{align}
We scale the extinction time $T^*$ of $u$ to infinity by letting
\be\label{eq:scaling0}
v(\rho,\theta,t)=\Big (\frac{T^*}{T^*-\tau}\Big)^{\frac{1}{p-1}} w(\rho,\theta,\tau), \quad  t =T^*\ln \Big(\frac{T^*}{T^*-\tau}\Big).
\ee
Then the fast diffusion equation in \eqref{eq:fde0} becomes
\begin{equation}\label{eq:fde1}
\begin{split}
\frac{\pa }{\pa t} v^{p}&=\frac{\pa^2}{\pa \rho^2}v+\Delta_{\mathbb{S}^{n-1}} v +a \frac{\pa}{\pa \rho } v+ bv +\frac{p}{(p-1)T^*} v^p\quad \mbox{on }\Sigma\times (0,\infty),\\
v(\rho,\theta,0)&=f(e^\rho\theta)^{\frac 1p}.
\end{split}
\end{equation}
Therefore, the extinction profiles of the solution $u$ of \eqref{eq:fde0} as $\tau\to T^*$ can be deduced from the asymptotic profiles of the solution $v$ of \eqref{eq:fde1} as $t\to\infty$, which are expected to be the solutions of the stationary equation. 

\subsection{Singular stationary solutions}\label{thm:criticalnoboundary}
The stationary equation of \eqref{eq:fde1} is
\begin{equation}\label{eq:stationary1}
\frac{\pa^2}{\pa \rho^2}v+\Delta_{\mathbb{S}^{n-1}} v +a \frac{\pa}{\pa \rho } v+ bv +\frac{p}{p-1} v^p=0\quad \mbox{on }(-\infty,\infty)\times \mathbb{S}^{n-1}.
\end{equation}
We suppose that $n\ge 3$. The following  results are known.

\begin{itemize}

\item[(i).] If $\frac{n}{n-2}<p<\frac{n+2}{n-2}$, and $v$ is uniformly bounded away from $0$ and $\infty$, then $v$ has to be a constant function, that is, 
\begin{equation}\label{eq:stationaryconstant}
 v\equiv \left(\frac{2\big((n-2)p-n\big)}{p(p-1)}\right)^{\frac{1}{p-1}}.
 \end{equation}
This result was proved by Gidas-Spruck \cite{GS}.

\item[(ii).]  Suppose that $v$ does not depend on the $\rho$-variable, $\frac{n}{n-2}<p<\frac{n+1}{n-3}$ if $n>3$ and $\frac{n}{n-2}<p<\infty$ if $n=3$.  Then $v$ has to the constant function given in \eqref{eq:stationaryconstant}. This was proved by Bidaut-V\'eron-V\'eron \cite{BVeron}.

\item[(iii).] Suppose that $v$ does not depend on the $\rho$-variable, $n>3$  and $p=\frac{n+1}{n-3}$, then 
  \begin{equation}\label{eq:bubble}
 v(\theta)= v_{\theta_0,\lambda}(\theta):=\left(\frac{(n-1)(n-3)}{n+1}\right)^{\frac{n-3}{4}}\cdot \left(\frac{\sqrt{\lda^2-1}}{\lda-\cos(\mbox{dist}(\theta,\theta_0))}\right)^{\frac{n-3}{2}}
  \end{equation}
for some $\theta_0\in\mathbb{S}^{n-1}$ and $\lda>1$, where $\mbox{dist}(\theta,\theta_0)$ is the geodesic distant between $\theta$ and $\theta_0$ on $\mathbb{S}^{n-1}$. This result was proved by Obata \cite{Obata}.

\item[(iv).] If $p=\frac{n+2}{n-2}$ and $\liminf_{\rho\to\infty}\min_{\theta\in\mathbb{S}^{n-1}}v(\rho,\theta)>0$, then $v$ depends only on the $\rho$-variable, and the equation \eqref{eq:stationary1} becomes an ODE. This ODE has a first integral, and all its solutions can be classified by the usual phase plane techniques, that is, either $v$ is a constant function, or $v$ is periodic in the $\rho$-variable with period $\ell>\frac{2\pi}{\sqrt{n-2}}$. These periodic solutions are usually called  the Fowler solutions. For $\ell>\frac{2\pi}{\sqrt{n-2}}$, there is only one Fowler solution with (minimal) period $\ell$. These results can be found in  Caffarelli-Gidas-Spruck \cite{CGS}, and also in Mazzeo-Pacard \cite{MP}.

\end{itemize}

We will use the above results (ii), (iii) and (iv) in our theorems.

\subsection{Main results}

Recall that we will assume the initial data of \eqref{eq:fde0} are of the form \eqref{eq:initialdata}. In the cylindrical coordinate system $\theta=x/|x|\in\mathbb{S}^{n-1}$, $\rho\in\R$, $|x|=e^\rho$, if we write
\[
f_0(\rho,\theta)=f(e^\rho\theta),
\]
then
\begin{equation}\label{eq:initialinradial}
u_0(x)=|x|^{-\frac{2}{1-m}} f_0\left(\log|x|,\frac{x}{|x|}\right)=|x|^{-\frac{2}{1-m}} f_0\left(\rho,\theta\right).
\end{equation}
 We obtain the following asymptotic behavior of solutions to \eqref{eq:fde0} near the vanishing time $T^*$.

\begin{thm}\label{thm:criticalglobal}
Let $n\ge 3$ and $m=\frac{n-2}{n+2}$. Suppose $f_0(\rho,\theta)$ is a smooth positive function on $\R\times\mathbb{S}^{n-1}$, and is periodic in the $\rho$-variable with period $\ell>0$. Let $u$ be the solution of \eqref{eq:fde0} with initial data \eqref{eq:initialinradial}. Let $T^*$ be its extinction time. Then 
\[
\left(\frac{1}{T^*-t}\right)^\frac{1}{1-m}u(x,t) \to \bar v^\frac{1}{m}\left(\log|x|\right)\cdot |x|^{-\frac{2}{1-m}}\quad\mbox{in }C^2_{loc}(\R^n\setminus\{0\})\ \mbox{as }t\to T^*.
\]
where $\bar v$ is a  Fowler solution of \eqref{eq:stationary1}. Moreover, if $\ell\le 2\pi/\sqrt{n-2}$, then $\bar v=\left(\frac{n-2}{\sqrt{n+2}}\right)^{(n-2)/4}$.
\end{thm}

\begin{thm}\label{thm:subcriticalglobal20}
Let $n> 3$ and $m=\frac{n-3}{n+1}$. Suppose $f_0(\rho,\theta)$ is independent of the $\rho$-variable, and is a smooth positive function in the $\theta$-variable. Let $u$ be the solution of \eqref{eq:fde0} with initial data \eqref{eq:initialinradial}. Let $T^*$ be its extinction time. Then there exist $\theta_0\in\mathbb{S}^{n-1}$ and $\lda>1$ such that
\[
\left(\frac{1}{T^*-t}\right)^\frac{1}{1-m}u(x,t) \to  v_{\theta_0,\lambda}^\frac{1}{m}\left(\frac{x}{|x|}\right)\cdot |x|^{-\frac{2}{1-m}}\mbox{ in }C^2_{loc}(\R^n\setminus\{0\})\ \mbox{as }t\to T^*,
\]
where $ v_{\theta_0,\lambda}$ is  given in \eqref{eq:bubble}
\end{thm}

\begin{thm}\label{thm:subcriticalglobal}
Let $n\ge 3$ and $\frac{n-3}{n+1}<m<\frac{n-2}{n}$. Suppose $f_0(\rho,\theta)$ is independent of the $\rho$-variable, and is a smooth positive function in the $\theta$-variable. Let $u$ be the solution of \eqref{eq:fde0} with initial data \eqref{eq:initialinradial}. Let $T^*$ be its extinction time. Then 
\[
\left(\frac{1}{T^*-t}\right)^\frac{1}{1-m}u(x,t) \to \left(\frac{2m(n-2-nm)}{1-m}\right)^{\frac{1}{1-m}}\cdot |x|^{-\frac{2}{1-m}}\mbox{ in }C^2_{loc}(\R^n\setminus\{0\})\ \mbox{as }t\to T^*.
\]
\end{thm}

\begin{rem}
Because we assume in the above three theorems that $v_0$ is either independent of the $\rho$-variable, or is periodic in the $\rho$-variable, the $C^2_{loc}$ convergence in these three theorems are sufficient to capture the global convergence. 
\end{rem}

\begin{rem}
The convergence rates in all of Theorems \ref{thm:criticalglobal}, \ref{thm:subcriticalglobal20} and \ref{thm:subcriticalglobal} will be at least  $|\ln(T^*-t)|^{-\gamma}$ for some $\gamma>0$ depending only on $n$, $m$ and $f_0$. See Theorems \ref{thm:criticalglobal2}, \ref{thm:subcriticalglobal22} and \ref{thm:subcriticalglobal2}.

\end{rem}

Our above theorems essentially follow from the observation that if we rewrite the equation \eqref{eq:fde0} under cylindrical coordinates \eqref{eq:cylinercoordinate}, then they will be reduced to obtaining the vanishing behavior of fast diffusion equations \eqref{eq:fdew00} on compact manifolds without boundary. Much of the analysis for our fast diffusion equations on compact manifolds will be similar to those in bounded domains  (more similar to Neumann problems than to Dirichlet problems) in the Euclidean space, where we will use arguments from literatures including Berryman-Holland \cite{BH}, Bonforte-V\'azquez \cite{BV}, del Pino-S\'aez \cite{del} and Simon \cite{Simon}. We also refer to Bonforte-Grillo-V\'azquez \cite{BGV2}, Akagi \cite{Akagi}, Bonforte-Figalli \cite{BF}, Jin-Xiong \cite{JX19,JX20} and  the references therein for more study on the extinction profiles of fast diffusion equations on bounded domains.

This paper is organized as follows. In the next section, we will rewrite the equations in cylindrical coordinates and formulate some equivalent problems. In Section \ref{sec:integralbound}, we obtain the energy estimates. In Section \ref{sec:criticalnoboundary}, we prove Theorems \ref{thm:criticalglobal} and \ref{thm:subcriticalglobal20}. In Section \ref{sec:subcriticalnoboundary}, we prove Theorem \ref{thm:subcriticalglobal}.

\bigskip

\noindent \textbf{Acknowledgement:} Part of this work was completed while the second named author was visiting the Hong Kong University of Science and Technology and Rutgers University, to which he is grateful for providing  the very stimulating research environments and supports. Both authors would like to thank Professor YanYan Li for his interests and constant encouragement. Finally, we would like to thank the anonymous referee for his/her careful reading of the paper, for sharing expertise and references on fast diffusion equations,  and for invaluable suggestions that greatly improved the presentation of the paper.

\section{Reformulations on product manifolds}

We have noted that finding the extinction profiles of solutions to \eqref{eq:fde0} is reduced to describing the asymptotic behavior of solutions to \eqref{eq:fde1}  as $t\to\infty$. We will show that the solutions of \eqref{eq:fde1}  will converge to their stationary solutions under the assumptions in our theorems on the initial data and proper range of the exponent $p$.

Under the assumptions of Theorem \ref{thm:criticalglobal}, we know for  \eqref{eq:fdew00} that $p=\frac{n+2}{n-2}$ and $w_0$ is periodic in $\rho$ with period $\ell$. Therefore, $a=0$, and due to the well-posedness of \eqref{eq:fde0}, $w$ defined by \eqref{eq:cylinercoordinate} satisfies
\begin{equation}\label{eq:fde1critical0}
\begin{split}
\frac{\pa }{\pa t} w^{p}&=\frac{\pa^2}{\pa \rho^2}w+\Delta_{\mathbb{S}^{n-1}} w + bw \quad \mbox{on }\Sigma_\ell\times (0,T^*),\\
w(\rho,\theta,0)&=f_0(\rho,\theta)^{\frac 1p},
\end{split}
\end{equation}
where $f_0(\rho,\theta)=f(e^\rho\theta)$, $b$ is given in \eqref{eq:bconstant}, and $$\Sigma_\ell=(\R/\ell\mathbb{Z})\times \mathbb{S}^{n-1}.$$  
Note that for the choice of $b$ in \eqref{eq:bconstant}, $b<0$ if $p>\frac{n}{n-2}$. 
The rescaled solution $v$ defined by \eqref{eq:scaling0} satisfies
\begin{equation}\label{eq:fde1critical}
\begin{split}
\frac{\pa }{\pa t} v^{p}&=\frac{\pa^2}{\pa \rho^2}v+\Delta_{\mathbb{S}^{n-1}} v + bv +\frac{p}{(p-1)T^*} v^p\quad \mbox{on }\Sigma_\ell\times (0,\infty),\\
v(\rho,\theta,0)&=f_0(\rho,\theta)^{\frac 1p}.
\end{split}
\end{equation}
 We have
 \begin{thm}\label{thm:criticalglobal2}
Let $n\ge 3$ and $p=\frac{n+2}{n-2}$. Suppose $f_0(\rho, \theta)$ is a smooth positive function on $\Sigma_\ell$, and $w$ is a solution of \eqref{eq:fde1critical0} with $b$ given in \eqref{eq:bconstant} and extinction time $T^*$. Let $v$ be the rescaled solution defined in \eqref{eq:scaling0} that satisfies \eqref{eq:fde1critical}. Then $v(\cdot,t)$ converges smoothly  to a stationary solution $\bar v$ as $t\to\infty$. If $\ell\le 2\pi/\sqrt{n-2}$, then $\bar v=\left(\frac{T^*(n-2)^2}{n+2}\right)^{(n-2)/4}$. Moreover, there exist $C>0$ and $\gamma>0$, both of which depend only on $n$ and $f_0$, such that
\[
\|v(\cdot,t)-\bar v\|_{C^2(\Sigma_\ell)}\le Ct^{-\gamma}\quad\mbox{for all }t>1.
\]
\end{thm}

The Fowler solutions in Theorem \ref{thm:criticalglobal2} were presented earlier in Section \ref{thm:criticalnoboundary} (iv).

Under the assumptions in Theorems \ref{thm:subcriticalglobal20} and \ref{thm:subcriticalglobal}, we know for \eqref{eq:fdew00} that $w_0$ is independent of  $\rho$. Therefore, $w$ defined in \eqref{eq:cylinercoordinate} satisfies
\begin{equation}\label{eq:fde1subcritical0}
\begin{split}
\frac{\pa }{\pa t} w^{p}&=\Delta_{\mathbb{S}^{n-1}} w + bw \quad \mbox{on }\mathbb{S}^{n-1}\times (0,T^*),\\
w(\rho,\theta,0)&=f_0(\theta)^{\frac 1p},
\end{split}
\end{equation}
where $f_0(\theta)=f_0(1,\theta)$, $b$ is given in \eqref{eq:bconstant}, and the rescaled solution $v$ defined in \eqref{eq:scaling0} satisfies
\begin{equation}\label{eq:fde1subcritical}
\begin{split}
\frac{\pa }{\pa t} v^{p}&=\Delta_{\mathbb{S}^{n-1}} v + bv +\frac{p}{(p-1)T^*} v^p\quad \mbox{on }\mathbb{S}^{n-1}\times (0,\infty),\\
v(\theta,0)&=f_0(\theta)^{\frac 1p}.
\end{split}
\end{equation}
If $p=\frac{n+1}{n-3}$ with $n>3$, then \eqref{eq:fde1subcritical} is a variant of the Yamabe flow on $\mathbb{S}^{n-1}$ which does not preserve the volume. It has been proved by del Pino-S\'aez \cite{del}  that 

\begin{thm}[del Pino-S\'aez \cite{del}]\label{thm:subcriticalglobal22}
Let $n> 3$ and $p=\frac{n+1}{n-3}$. Suppose $f_0(\theta)$ is a smooth positive function on $\mathbb{S}^{n-1}$, and $w$ is a solution of \eqref{eq:fde1subcritical0} with $b$ given in \eqref{eq:bconstant} and extinction time $T^*$. Let $v$ be the rescaled solution defined in \eqref{eq:scaling0} that satisfies \eqref{eq:fde1subcritical}. Then $v(\cdot,t)$ converges smoothly to $(T^*)^{\frac{1}{p-1}}\bar v$ as $t\to\infty$, where  $\bar v$ is given in \eqref{eq:bubble}.  Moreover, there exist $C>0$ and $\gamma>0$, both of which depend only on $n$ and $f_0$, such that
\[
\|v(\cdot,t)-(T^*)^{\frac{1}{p-1}}\bar v\|_{C^2(\mathbb{S}^{n-1})}\le Ct^{-\gamma}\quad\mbox{for all }t>1.
\]
\end{thm}
Although the decay rate $t^{-\gamma}$ in Theorem \ref{thm:subcriticalglobal22} is not stated in \cite{del}, it can be obtained in a similar way to that of Theorem \ref{thm:criticalglobal2}.

For subcritical exponents, we have
\begin{thm}\label{thm:subcriticalglobal2}
Let $n\ge 3$. Let $\frac{n}{n-2}<p<\frac{n+1}{n-3}$ if $n>3$, and $\frac{n}{n-2}<p<\infty$ if $n=3$. Suppose $f_0(\theta)$ is a smooth positive function on $\mathbb{S}^{n-1}$, and $w$ is a solution of \eqref{eq:fde1subcritical0} with $b$ given in \eqref{eq:bconstant} and extinction time $T^*$. Let $v$ be the rescaled solution defined in \eqref{eq:scaling0} that satisfies \eqref{eq:fde1subcritical} with $b$ given in \eqref{eq:bconstant}. Then $v(\cdot,t)$ converges smoothly to the constant $\left(\frac{2T^*(p(n-2)-n)}{p(p-1)}\right)^{1/(p-1)}$ as $t\to\infty$. Moreover, there exist $C>0$ and $\gamma>0$, both of which depend only on $n$, $p$ and $f_0$, such that
\[
\left\|v(\cdot,t)-\left(\frac{2T^*(p(n-2)-n)}{p(p-1)}\right)^{1/(p-1)}\right\|_{C^2(\mathbb{S}^{n-1})}\le Ct^{-\gamma}\quad\mbox{for all }t>1.
\]

\end{thm}

Due to the well-posedness of \eqref{eq:fde0}, Theorems \ref{thm:criticalglobal}, \ref{thm:subcriticalglobal20} and \ref{thm:subcriticalglobal} follows from Theorems \ref{thm:criticalglobal2}, \ref{thm:subcriticalglobal22} and \ref{thm:subcriticalglobal2}, respectively.

Note that the differential operator $\frac{\partial^2}{\partial \rho^2}+\Delta_{\mathbb{S}^{n-1}}$ is the Beltrami-Laplace operator of the product manifold $(\R/\ell\mathbb{Z})\times \mathbb{S}^{n-1}$. Therefore, the equations \eqref{eq:fde1critical0} and \eqref{eq:fde1subcritical0} can be written in a unified form:
\begin{equation}\label{eq:fdemanifold}
\begin{split}
\frac{\pa }{\pa t} w^{p}&=\Delta_{g} w + bw\quad \mbox{on }M\times (0,\infty),\\
w(\cdot,0)&=w_0,
\end{split}
\end{equation}
where $(M,g)$ is a compact manifold without boundary, $\Delta_{g} $ is its Beltrami-Laplace operator, and $b$ is a smooth function on $M$ such that the operator $-\Delta_{g}-b$ is coercive, that is, there exists $c_0>0$ such that
\be\label{eq:coercive}
\int_{M}(|\nabla_g u|^2-bu^2)\,\ud vol_g\ge  c_0 \int_{M}u^2\,\ud vol_g\quad\mbox{for all }u\in C^\infty(M).
\ee
Then by the Sobolev inequality, we will have that (see the proof of \eqref{eq:decreasing} for details)
\[
\frac{\ud}{\ud t} \left(\int_M w(\cdot,t)^{p+1}\,\ud vol_g\right)^{\frac{p-1}{p+1}}\le - C
\]
for some $C>0$ depending only on $n,p,M,g$ and the $c_0$ in \eqref{eq:coercive}. Therefore, solutions of \eqref{eq:fdemanifold}  extinct after a fine time $T^*$. Under the change of variables 
\be\label{eq:scaling0m}
v(x,t)=\Big (\frac{T^*}{T^*-\tau}\Big)^{\frac{1}{p-1}} w(x,\tau), \quad  t =T^*\ln \Big(\frac{T^*}{T^*-\tau}\Big),
\ee 
we have
\begin{equation}\label{eq:fdemanifold1}
\begin{split}
\frac{\pa }{\pa t} v^{p}&=\Delta_{g} v + bv+\frac{p}{(p-1)T^*}v^p\quad \mbox{on }M\times (0,\infty),\\
v(\cdot,0)&=w_0.
\end{split}
\end{equation}
Theorem \ref{thm:subcriticalglobal2} can be slightly generalized to
\begin{thm}\label{thm:subcriticalglobal3}
Let $(M,g)$ be an $n$-dimensional  smooth compact manifold without boundary. Let $1<p<\frac{n+2}{n-2}$ if $n\ge 3$, and $1<p<\infty$ if $n=1,2$. Suppose $w_0$ is a smooth positive function on $M$,  and $b$ is a smooth function on $M$ such that \eqref{eq:coercive} holds. Let $w$ be the solution of \eqref{eq:fdemanifold} with initial data $w_0$, $T^*$ be its extinction time, and $v$ be defined by \eqref{eq:scaling0m}. Then $v(\cdot,t)$ converges smoothly to a stationary solution $\bar v$ of \eqref{eq:fdemanifold1} as $t\to\infty$.  Moreover, there exist $C>0$ and $\gamma>0$, both of which depend only on $n$, $p$, $b$, $M$, $g$, $c_0$ and $w_0$, such that
\[
\|v(\cdot,t)-\bar v\|_{C^2(M)}\le Ct^{-\gamma}\quad\mbox{for all }t>1.
\]

\end{thm}

The proof of Theorem \ref{thm:criticalglobal2} is given in Section \ref{sec:criticalnoboundary}. The proof of Theorems \ref{thm:subcriticalglobal2} and \ref{thm:subcriticalglobal3} are given in  Section \ref{sec:subcriticalnoboundary}.

Fast diffusion equations on general noncompact manifolds have been studied in  Bonforte-Grillo-V\'azquez \cite{BGV08,BGV}, Bianchi-Setti \cite{BS18}, Grillo-Muratori-Punzo \cite{GMP}, etc.

\section{Some integral bounds}\label{sec:integralbound}

In this section, we first follow the arguments of Berryman-Holland \cite{BH} to obtain some integral bounds for the solutions of \eqref{eq:fdemanifold}. We assume $1<p\le \frac{n+2}{n-2}$ if $n>3$, and $1<p<\infty$ if $n=1,2$. Denote
\[
L=-(\Delta_g+b),
\]
and we assume $b$ is a smooth function on $M$ satisfying \eqref{eq:coercive}. Let $w$ be the solution of \eqref{eq:fdemanifold} with smooth and positive initial data $w_0$, and let $T^*$ be its extinction time. Let
\begin{equation}\label{eq:Sobolevquotient}
H(t)=H(w(\cdot,t))= \frac{\int_{M} w Lw\,\ud vol_{g}}{(\int_{M} w^{p+1}\,\ud vol_{g} )^{2/(p+1)}  },
\end{equation}
and $H_0=H(0)=\frac{\int_{M} w_0 Lw_0\,\ud vol_{g}}{(\int_{M} w_0^{p+1}\,\ud vol_{g} )^{2/(p+1)}  }. $ We have the following integral estimates for $w$.

\begin{lem}\label{lem:flow-1}  
There exists a positive constant $C$ depending only on $n,p,M,g$ and the $c_0$ in \eqref{eq:coercive} such that for all $0\le t\le T^*$, we have 
\begin{align*}
C (T^*-t)\le \left(\int_{M} w(\cdot, t)^{p+1}\,\ud vol_g\right)^{\frac{p-1}{p+1}} &\le  \frac{p+1}{p}(T^*-t) H_0.
\end{align*}
\end{lem}

\begin{proof} The proofs are identical to those in \cite{BH} in Euclidean spaces. We include them here for completeness. 

Since $M$ is compact and $w_0$ is smooth and positive, we know that $w$ is smooth and positive before its extinction time.
By the equation of $w$, we have 
\begin{align}
\frac{\ud }{\ud t} \int_M wLw\,\ud vol_{g}&= -\frac{2}{p} \int_{M} \frac{(Lw)^2}{w^{p-1}}\,\ud vol_{g},\nonumber\\
 \frac{\ud }{\ud t} \int_{M} w^{p+1}\,\ud vol_{g}&= -\frac{p+1}{p} \int_{M} wLw\,\ud vol_{g}. \label{eq:lem:flow-xaux}
\end{align}
Therefore,  
\begin{align}\label{eq:Hdecreasing}
\frac{\ud }{\ud t}H(t)=\frac{2}{p}H(w(\cdot,t)) \left(\frac{\int_{M} wLw\,\ud vol_{g}}{\int_{M} w^{p+1}\,\ud vol_{g}}-\frac{\int_{M}\frac{(Lw)^2}{u^{p-1}}\,\ud vol_{g}}{\int_{M} wLw\,\ud vol_{g}}\right)\le 0, 
\end{align}
where in the last inequality we used the Cauchy-Schwarz inequality.
\[
\left(\int_{M}  wLw\,\ud vol_{g}\right )^2 \le \int_{M} w^{p+1}\,\ud vol_{g} \int_{M}\frac{(Lw)^2}{w^{p-1}} \,\ud vol_{g}.
\]

Let 
\[
\zeta(t)=\left(\int_{M} w(\cdot,t)^{p+1}\,\ud vol_g\right)^{\frac{p-1}{p+1}}. 
\]
Then one can verify that
\begin{equation}\label{eq:derivative}
\zeta'(t)=-\frac{p-1}{p} H(t). 
\end{equation}
Since $b$ satisfies \eqref{eq:coercive}, $1<p\le \frac{n+2}{n-2}$ if $n>3$, and $1<p<\infty$ if $n=1,2$, by the Sobolev inequality and  H\"older's inequality, we have 
\[
H(t)\ge C>0
\]
for some positive constant $C$ depending only on $n,p,M,g$ and the $c_0$ in \eqref{eq:coercive}. Therefore, it follows from \eqref{eq:derivative} that 
\begin{equation}\label{eq:decreasing}
\zeta'(t)\le-C
\end{equation}
for some positive constant $C$ depending only on $n,p,M,g$ and the $c_0$ in \eqref{eq:coercive}.  Integrating the above inequality, we have 
\[
\left(\int_{M} w(\cdot,t')^{p+1}\,\ud vol_g\right)^{\frac{p-1}{p+1}}-\left(\int_{M} w(\cdot,t)^{p+1}\,\ud vol_g\right)^{\frac{p-1}{p+1}}\le -C(t'-t)
\]
for $T^*>t'\ge t$. The first inequality of this lemma follows immediately from sending $t'\to T^*$.

Multiplying $-w_t$ to \eqref{eq:fdemanifold} and integrating over $M$, we have 
\begin{align}\label{eq:energingdreasing}
\frac{d}{dt}\int_M wLw\,\ud vol_{g} \le 0.
\end{align}
Then, by integrating \eqref{eq:lem:flow-xaux} from $t$ to $T^*$ and using the above inequality, we have
\begin{align*}
\int_M w^{p+1}(\cdot,t)\,\ud vol_{g}&=\frac{p+1}{p}\int_t^{T^*} \int_M wLw\,\ud vol_{g}\ud t\\
&\le \frac{p+1}{p}(T^*-t) \int_M w(\cdot,t)Lw(\cdot,t)\,\ud vol_{g}\\
&=\frac{p+1}{p}(T^*-t) H(t) \left(\int_{M} w^{p+1}(\cdot,t)\,\ud vol_{g} \right)^{2/(p+1)}\\
&\le   \frac{p+1}{p}(T^*-t) H_0 \left(\int_{M} w^{p+1}(\cdot,t)\,\ud vol_{g} \right)^{2/(p+1)},
\end{align*}
where we used \eqref{eq:energingdreasing} in the first inequality and \eqref{eq:Hdecreasing} in the last inequality. Hence, the second inequality of this lemma follows.
\end{proof}

By setting $t=0$ in Lemma \ref{lem:flow-1}, we have 
\begin{align}\label{eq:estiamteofT}
\frac{p}{(p+1)H_0} \left(\int_{M} w_0^{p+1}\,\ud vol_g\right)^{\frac{p-1}{p+1}} \le T^* \le \frac{1}{C}\left(\int_{M} w_0^{p+1}\,\ud vol_g\right)^{\frac{p-1}{p+1}}.
\end{align}
for some $C>0$ depending only on $n,p,M,g$ and the $c_0$ in \eqref{eq:coercive}. It follows from Lemma \ref{lem:flow-1} that the function $v$ in each of Theorems \ref{thm:criticalglobal2}, \ref{thm:subcriticalglobal22}, \ref{thm:subcriticalglobal2} and \ref{thm:subcriticalglobal3} satisfies
\begin{align}\label{eq:vLpbound} 
\frac 1C_0&\le\|v(\cdot,t)\|_{L^{p+1}} \le C_0\quad\forall\ t\in [0,\infty)
\end{align}
for some positive constant $C_0$ which depends only on $n,p,M,g$, $c_0$ in \eqref{eq:coercive}, $\int_{M} w_0^{p+1}\,\ud vol_g$ and $\int_{M} w_0 Lw_0\,\ud vol_{g}$.

\section{Critical cases}\label{sec:criticalnoboundary}
In this section, we will prove Theorem \ref{thm:criticalglobal2}, and consequently, Theorem \ref{thm:criticalglobal}.

\subsection{Uniform bounds}

An intermediate step of proving the convergence in Theorem \ref{thm:criticalglobal} is the following so-called global Harnack inequality (see Bonforte-Vazquez \cite{BVjfa} for a similar case of having a lower and upper bound in terms of Barenblatt profiles, and  see also Bonforte-Simonov \cite{BS20} for a complete characterization of the maximal set of initial data that produces solutions which are pointwisely trapped between two Barenblatt solutions):

\begin{prop}\label{prop:globalHarnack}
Assume all the assumptions in Theorem \ref{thm:criticalglobal}. Then there exists a constant $C>0$ depending only on $n$ and $f_0$ such that
\[
\frac{1}{C} (T^*-t)^{\frac{1}{1-m}} |x|^{-\frac{2}{1-m}}\le u(x,t)\le C (T^*-t)^{\frac{1}{1-m}} |x|^{-\frac{2}{1-m}}
\]
for all $(x,t)\in (\R^n\setminus\{0\})\times(0,T^*)$.
\end{prop}

To prove this proposition, we will first prove the uniform lower and upper boundedness of the solution $v$ in  Theorem \ref{thm:criticalglobal2}.

For $p=\frac{n+2}{n-2}$, the equation \eqref{eq:fde1critical} has conformal invariance. We will make use of this invariance to obtain a Harnack inequality, and consequently, uniform estimates.

Denote
\[
u_{x,\lda}(y)= \left(\frac{\lda}{|y-x|}\right)^{n-2} u\left(x+ \frac{\lda^2 (y-x)}{|y-x|^2}\right)
\]
as the Kelvin transform of $u$ with respect to the sphere $\pa B_{\lda}(x)$.

\begin{lem}\label{lem:aux} Let $u\in C^0(B_r), r > 0,$ be a positive function. Assume that
\[
 u_{x,\lda}(y)\le  u(y) \quad \mbox{for any }B_\lda(x) \subset  B_r  \mbox{ and }y\in B_r\setminus B_\lda(x).
 \]
Then $\ln u$ is locally Lipschitz in $B_r$ and
\[
|\nabla \ln u|\le \frac{n-2}{r-|x|} \quad \mbox{for }a.e. ~x\in B_r.
\]
\end{lem}

The above lemma was first proved in Lemma A.2 of Li-Li \cite{LL} assuming that $w$ is $C^1$. This condition was weakened to be $C^0$ in Lemma 2 of Li-Nguyen \cite{LN}.

\begin{prop} \label{prop:harnack} Let $n\ge 3$, $u\in C^2((\R^n\setminus \{0\})\times [0,T))$ be a positive  solution of
\begin{align*}\frac{\pa }{\pa t}u(x,t)^{\frac{n+2}{n-2}}&=\Delta u(x,t) +b(t) u^{\frac{n+2}{n-2}} \quad \mbox{in }(\R^n\setminus \{0\})\times (0,T),\\
u(x,0)&=u_0(x)> 0 \quad \mbox{in }\R^n\setminus \{0\},
\end{align*}
where $b(t)\in C([0,T))$. Suppose that for all $t\in [0,T)$,
\be\label{eq:cylinderdata-1}
\inf_{\R^n\setminus \{0\}} |x|^{\frac{n-2}{2}}u(x,t)>0\quad\mbox{and}\quad\sup_{\R^n\setminus \{0\}} |x|^{\frac{n-2}{2}}u(x,t)<\infty.
\ee
Let $A>0$ be such that
\begin{equation}\label{eq:lowerandupperbound}
\inf_{\R^n\setminus \{0\}} |x|^{\frac{n-2}{2}}u_0(x)\ge\frac{1}{A}, \ \sup_{\R^n\setminus \{0\}} |x|^{\frac{n-2}{2}}u_0(x)\le A\ \mbox{and}\ \sup_{\R^n\setminus \{0\}}|x|\cdot |\nabla_x \ln u_0(x)| \le A.
\end{equation}
Then there exists $C>0$ depending only on $n$ and $A$ such that 
\[
|\nabla_x \ln u(x,t)| \le \frac{C}{|x|} \quad \mbox{for all }x\in\R^n\setminus\{0\}\mbox{ and }\ 0<t<T.
\]
\end{prop}

\begin{proof} Let $u^\mu(x,t)= \mu^{\frac{n-2}{2}} u(\mu x, t)$ with $\mu>0$. Then $u^\mu$ satisfies all the assumptions of the proposition. Therefore, we only need to prove the proposition for $|x|=1$.

Arbitrarily fix $|x_0|=1$. We are going to show that there exists $r\in (0,1/4)$, depending only on $n,\sigma$ and $A$, but independent of the choice of $x_0$, such that for each $t\in [0,T)$ there holds
\be \label{eq:movingsphere-1}
u_{x,\lda}(y,t)\le  u(y,t) \quad \mbox{for any }B_\lda(x) \subset  B_r(x_0)  \mbox{ and }y\in \R^n\setminus B_\lda(x),\ y\neq 0.
\ee
Then the proposition will follow immediately from Lemma \ref{lem:aux}.

Define
\[
r_0=\min\left\{\frac{n-2}{9A},\frac14\right\}.
\]
Then for every $x\in B_{1/2}(x_0)$, $0<r<r_0$ and $\theta\in\mathbb{S}^{n-1}$, we have $|x+r\theta|\ge 1/4$, and thus, 
\begin{align*}
\frac{\ud }{\ud r}(r^{\frac{n-2}{2}} u_0( x+r\theta))&\ge r^{\frac{n-4}{2}}u_0( x+r\theta)\left(\frac{n-2}{2}-\frac{r|\nabla u_0( x+r\theta)|}{u_0( x+r\theta)}\right)\\
&\ge r^{\frac{n-4}{2}}u_0( x+r\theta)\left(\frac{n-2}{2}-4Ar\right)>0.
\end{align*}
This implies that 
\[
(u_0)_{x,\lda}(y)\le  u_0(y) \quad \mbox{for any } 0<\lda<|y-x|<r_0.
\]
It is elementary to verify that there exists $\bar\lda=\bar\lda(A,r_0)\in(0,1/4)$ such that
\[
|y|\bar\lda^2\le |x-y|^2A^{-\frac{4}{n-2}}\left(\frac{1}{2}-\bar\lda\right)\quad\mbox{for all }|y-x|\ge r_0.
\]
Then for every $0<\lda\le\bar\lda$ and $|y-x|\ge r_0$, we have
\[
|y|\lda^2\le |x-y|^2 A^{-\frac{4}{n-2}} \left|x+\frac{\lda^2(y-x)}{|y-x|^2}\right|,
\]
from which, together the first two conditions in \eqref{eq:lowerandupperbound}, it follows that
\[
(u_0)_{x,\lda}(y)\le  u_0(y) \quad \mbox{if } 0<\lda<\bar\lda\quad\mbox{and}\quad  |y-x|\ge r_0,\ y\neq 0.
\]
Hence, 
\begin{equation}\label{eq:movingsphere}
(u_0)_{x,\lda}(y)\le  u_0(y) \quad \mbox{if } 0<\lda<\bar\lda\quad\mbox{and}\quad  |y-x|\ge \lda,\ y\neq 0.
\end{equation}
We are going to use the maximum principle to show that \eqref{eq:movingsphere} holds for all $u(\cdot,t)$. 

Notice that $u_{x,\lda}$ satisfies the same equation as $u$. Moreover, because of the first two conditions in \eqref{eq:lowerandupperbound}, it is elementary to see that for all $0<\lda<\bar\lda$, we have
\[
u_{x,\lda}(y,t)\le  u(y,t) \quad \mbox{when } |y|=\va\ \mbox{or}\  |y|=R
\] 
if $\va$ is sufficiently small and $R$ is sufficiently large. Now consider the equation of $u$ and $u_{x,\lda}$ in the region $\R^n\setminus(\{0\}\cup B_\lda(x))$ and apply the comparison principle, we have
\[
u_{x,\lda}(y,t)\le  u(y,t) \quad \mbox{if } 0<\lda<\bar\lda\quad\mbox{and}\quad  |y-x|\ge \lda,\ y\neq 0,
\]
from which \eqref{eq:movingsphere-1} follows with $r=\bar\lda$.
\end{proof}

\begin{prop} \label{prop:harnack-sub2} 
Let $v$ be as in Theorem \ref{thm:criticalglobal2}. There exists a constant $C$ depending only on $n,\ell$ and $f_0$ such that 
\[
\frac{1}{C} \le v(x,t)\le C \quad \mbox{for all }x\in\Sigma_\ell,\ t>0. 
\]
\end{prop}
\begin{proof}
Make the change of variables $r=|x|$, $\theta=\frac{x}{|x|}$, $\rho =\ln r$ and
\[
v(\rho,\theta,t)= r^{\frac{2}{p-1}} u(r\theta,t).
\]
Then $u$ satisfies Proposition \ref{prop:harnack}. Consider $(\rho_1,\theta_1), (\rho_2,\theta_2)\in [0,\ell)\times\mathbb{S}^{n-1}$, and denote $x_1=e^{\rho_1}\theta_1$ and $x_2=e^{\rho_2}\theta_2$. Then
\[
u(x_1,t)=e^{\frac{2\rho_1}{1-p}}v(\rho_1,\theta_1,t),\quad u(x_2,t)=e^{\frac{2\rho_2}{1-p}}v(\rho_2,\theta_2,t).
\]
Since $1\le|x_1|\le e^\ell$, $1\le|x_2|\le e^\ell$, it follows from Proposition \ref{prop:harnack} that 
\[
\left|\ln\frac{u(x_1,t) }{u(x_2,t)}\right|=|\ln u(x_1,t) - \ln u(x_2,t)|\le C,
\]
for some $C>0$ depending only on $n,\ell$ and $f_0$. Then it follows that 
\[
u(x_1,t) \le C u(x_2,t),
\]
and consequently,
\[
v(\rho_1,\theta_1,t)\le C v(\rho_2,\theta_2,t)
\]
for some $C>0$ depending only on $n,\ell$ and $f_0$. Since $(\rho_1,\theta_1)$ and $(\rho_2,\theta_2)$ are arbitrary,
then $v$ satisfies the following Harnack inequality on $\Sigma_\ell$: there exists $C>0$ depending only on $n, \ell$ and  $f_0$ such that
\begin{equation}\label{eq:Harnackcritical}
\sup_{\Sigma_\ell} v(\cdot,t)\le C \inf_{\Sigma_\ell} v(\cdot,t).
\end{equation}
The conclusion of this proposition follows from the above Harnack inequality and the integral estimate \eqref{eq:vLpbound}.
\end{proof}

Our Harnack inequality \eqref{eq:Harnackcritical} is an elliptic type Harnack inequality, that is, it holds on every time slice. In the proof, we used the conformal invariance of the equation for $p=\frac{n+2}{n-2}$ and the assumption of the periodicity (in $\rho$) of the solutions. Harnack type inequalities for solutions of general fast diffusion equations have been proved in Bonforte-Simonov \cite{BS2020Adv} and Bonforte-Dolbeault-Nazaret-Simonov \cite{BDNS2021}.

\begin{proof}[Proof of Proposition \ref{prop:globalHarnack}]
By the changes of variables in \eqref{eq:cylinercoordinate} and \eqref{eq:scaling0},  and the estimates of the extinction time $T^*$ in \eqref{eq:estiamteofT}, the lower and upper bounds of 
$u$ in Proposition \ref{prop:globalHarnack} are equivalent to the lower and upper bounds of $v$ in Proposition \ref{prop:harnack-sub2}.
\end{proof}

\subsection{Convergence}\label{sec:convergencecritical}
The right-hand side of \eqref{eq:fde1critical} is the negative gradient of the following functional
\begin{equation}\label{eq:energyJ}
J(v)=\frac{1}{2}\int_{\Sigma_\ell} |\nabla_g v|^2 \,\ud vol_g-\frac{b}{2} \int_{\Sigma_\ell} v^2\,\ud vol_g-\frac{p}{(p^2-1)T^*}\int_{\Sigma_\ell} v^{p+1}\,\ud vol_g.
\end{equation}
Once we have the uniform lower and upper bound of the solution, the convergence will essentially follow from the arguments in del Pino-S\'aez \cite{del} and Simon \cite{Simon}. 

First of all, we have that, denoting $J(t)=J(v(\cdot,t))$ with $v$ as in Theorem \ref{thm:criticalglobal2},
\begin{equation}\label{eq:Jdecrease}
\frac{d}{d t}J(t)=-\frac{4p}{(p+1)^2}\int_{\Sigma_\ell} \Big((v^{\frac{p+1}{2}})_t\Big)^2\,\ud vol_g\le 0.
\end{equation}
Secondly, we show that
\begin{lem}\label{lem:nonnegative}
$J(t)\ge 0$ for all $t>0$.
\end{lem}
\begin{proof}
Suppose that $J(t_0)< 0$ for some $t_0>0$. Then by \eqref{eq:Jdecrease}, $J(t)< 0$ for all $t\ge t_0$.
Let
\[
F(t)=\int_{\Sigma_\ell} v(\cdot,t)^{p+1}\,\ud vol_g.
\]
Then using the equation of $v$ and integration by parts, we have
\begin{align*}
\frac{p+1}{p}\frac{\ud F}{\ud t}(t)&=-\int_{\Sigma_\ell} |\nabla v|^2 \,\ud vol_g+ \int_{\Sigma_\ell} v^2\,\ud vol_g+\frac{p}{(p-1)T^*}\int_{\Sigma_\ell} v^\frac{2n}{n-2}\,\ud vol_g\\
&=-2J(t)+\frac{p+1}{pT^*}F(t)\\
&\ge \frac{p+1}{pT^*}F(t)\quad\mbox{for all }t\ge t_0.
\end{align*}
Hence
\[
F(t)\ge F(t_0)e^{\frac{t-t_0}{T^*}}.
\]
This contradicts the integral estimate \eqref{eq:vLpbound} if $t$ is sufficiently large.
\end{proof}

\begin{proof}[Proof of Theorem \ref{thm:criticalglobal2}:]
In this case, $b=-\frac{(n-2)^2}{4}$. By Proposition \ref{prop:harnack-sub2} and regularity theory, $v$ is bounded in all $C^k$ norms. Then there exists a sequence of time $\{t_j\}$ going to infinity such that $v(\cdot,t_j)\to \bar v$ as $j\to \infty$. We will show that $\bar v$ is a stationary solution to \eqref{eq:fde1critical} and $v(\cdot,t)\to\bar v$ as $t\to\infty$.

Integrating \eqref{eq:Jdecrease} from $t_j$ to $t_j+\tau$, and making use of the Cauchy-Schwarz inequality, we obtain
\[
\int_{\Sigma_\ell} |v(\cdot,t_j+\tau)^{\frac{p+1}{2}}-v(\cdot,t_j)^{\frac{p+1}{2}}|^2\,\ud vol_g\le \frac{(p+1)^2\tau}{4p} (J(t_j+\tau)-J(t_j)).
\]
Since $J(\cdot)$ is decreasing and nonnegative, $J(t)$ converges as $t\to\infty$. Therefore, for $\tau$ in bounded intervals, $v(\cdot,t_j+\tau)\to \bar v$ in $L^{p+1}$ uniformly in $\tau$ as $j\to\infty$, and hence $v(\cdot,t_j+\tau)\to \bar v$ in $C^{2}$ uniformly in $\tau$. Then we integrate the equation \eqref{eq:fde1} from $t_n$ to $t_n+1$, and obtain for every $\varphi \in C^\infty(\Sigma_\ell)$,
\begin{align*}
&\int_{\Sigma_\ell} (v(\cdot,t_j+1)^{p}-v(\cdot,t_j)^{p})\varphi\,\ud vol_g\\
&\quad = \int_{0}^{1} \int_{\Sigma_\ell} \left(\Delta_g v(\cdot,t_j+\tau)+bv(\cdot,t_j+\tau)+\frac{p}{(p-1)T^*}v(\cdot,t_j+\tau)^\frac{n+2}{n-2}\right)\varphi\,\ud vol_g\,\ud \tau.
\end{align*}
Sending $j\to\infty$, we have
\[
\int_{\Sigma_\ell} \left(\Delta_g \bar v+b\bar v+\frac{p}{(p-1)T^*}\bar v^\frac{n+2}{n-2}\right)\varphi\,\ud vol_g=0.
\]
Therefore, $\bar v$ is a stationary solution of \eqref{eq:fde1}.

Next, we show that $v(\cdot,t)\to \bar v$ in $C^2(\Sigma_\ell)$ as $t\to\infty$ (not just along a subsequence). This is essentially a consequence of the uniqueness result of Simon \cite{Simon} for negative gradient flows, and we provide its proof  as follows. Indeed, if we denote
\[
\nabla J (v)= \Delta_g  v+bv+\frac{n+2}{4T^*} v^\frac{n+2}{n-2},
\]
then it follows from Theorem 3 in \cite{Simon} that there exist $\theta\in (0,1/2)$ and $r_0>0$ such that for every $v\in C^{2,\alpha}(\Sigma_\ell)$ with $\|v-\bar v\|_{C^{2,\al}(\Sigma_\ell)}<r_0$, there holds
\begin{equation}\label{eq:Lojasiewicz}
\|\nabla J (v) \|_{L^2(\Sigma_\ell)}\geq \abs{J(v)-J(\bar v)}^{1-\theta}.
\end{equation}
This is an infinite dimensional generalization of the Lojasiewicz inequality. Also, from the equality in \eqref{eq:Jdecrease} and Cauchy-Schwarz inequality, there $c_0>0$ such that
\begin{equation}\label{eq:auxLojasiewicz}
-\frac{\ud}{\ud t} J(v(\cdot,t))\geq c_0 \|v_t(\cdot,t)\|_{L^2(\Sigma_\ell)}\|\nabla J(v(\cdot,t))\|_{L^2(\Sigma_\ell)}.
\end{equation}

For every $\va>0$ (suppose $\va<r_0$), we are going to find $\bar t=\bar t(\va)$ such that $\|v(\cdot,t)-\bar v\|_{C^{2,\al}(\Sigma_\ell)}<\va$ for all $t>\bar t$. From Proposition \ref{prop:harnack-sub2}, we know that $v$ satisfies a uniformly parabolic equation. Fix $\tau_0>0$. Since $v-\bar v$ satisfies
\[
pv^{p-1}(v-\bar v)_t=\Delta_{g} (v-\bar v) + b(v-\bar v) +\frac{p}{(p-1)T^*} (v^p-\bar v^p),
\]
that is,
\[
p(v-\bar v)_t=v^{1-p}\Delta_{g} (v-\bar v) + bv^{1-p}(v-\bar v) +\frac{pv^{1-p}}{(p-1)T^*} \frac{v^p-\bar v^p}{v-\bar v}(v-\bar v),
\]
by multiplying $v-\bar v$ on both sides and integrating by parts, and using Proposition \ref{prop:harnack-sub2}, we have 
\begin{align*}
&\frac{p}{2}\frac{\ud}{\ud t}\int_{\Sigma_\ell} (v-\bar v)^2\,\ud vol_g\\
&\le \int_{\Sigma_\ell} v^{1-p}(v-\bar v)  \Delta_{g} (v-\bar v)\, \ud vol_g+ C \int_{\Sigma_\ell} (v-\bar v)^2\,\ud vol_g\\
&\le \int_{\Sigma_\ell} v^{1-p} \left(\frac{1}{2}\Delta_{g} [(v-\bar v)^2] -  |\nabla_g(v-\bar v)|^2\right)\ud vol_g+ C \int_{\Sigma_\ell} (v-\bar v)^2\,\ud vol_g\\
&\le \int_{\Sigma_\ell} \frac{1}{2}  (v-\bar v)^2 \Delta_{g} (v^{1-p})\,\ud vol_g+ C \int_{\Sigma_\ell} (v-\bar v)^2\,\ud vol_g\\
&\le C \int_{\Sigma_\ell} (v-\bar v)^2\ud vol_g.
\end{align*}
Hence, by Gr\"onwall's inequality, we know that there exists $\delta_1(\va)>0$ such that if $\|v(\cdot,s)-\bar v\|_{L^2(\Sigma_\ell)}<\delta_1(\va)$, then 
\[
\|v(\cdot,t)-\bar v\|_{L^2(\Sigma_\ell)}<\va\quad\mbox{ for }t\in [s,s+\tau_0].
\]
Also, by the regularity estimate, there exists $\delta_2(\va)>0$ such that if $\|v(\cdot,s)-\bar v\|_{L^2(\Sigma_\ell)}<\delta_2(\va)$, then 
\[
\|v(\cdot,s+\tau_0)-\bar v\|_{C^{2,\alpha}(\Sigma_\ell)}<\va.
\]
From the continuity of $J$, there exists $\delta_3(\va)$ such that if $\|v-\bar v\|_{C^{2,\alpha}(\Sigma_\ell)}<\delta_3(\va)$ then 
\[
|J(v)-J(\bar v)|<\va.
\]
Let 
\[
\delta=\min\left\{\delta_1\left(\frac{1}{2}\delta_2(\va)\right),\delta_2\left(\delta_3\left(\left(\frac{c_0\theta\delta_2(\va)}{2}\right)^{\frac{1}{\theta}}\right)\right)\right\}.
\]
Choose $t_0$ sufficiently large such that $\|v(\cdot,t_0)-\bar v\|_{L^2(\Sigma_\ell)}<\delta$. This implies that
\[
\|v(\cdot,t)-\bar v\|_{L^2(\Sigma_\ell)}<\frac{1}{2}\delta_2(\va)\quad\mbox{ for }t\in [t_0,t_0+\tau_0],
\]
and consequently
\[
\|v(\cdot,t)-\bar v\|_{C^{2,\alpha}(\Sigma_\ell)}<\va\quad\mbox{ for }t\in [t_0+\tau_0,t_0+2\tau_0].
\]
Define
\[
T=\sup\{t: \|v(\cdot,s)-\bar v\|_{C^{2,\alpha}(\Sigma_\ell)}<\va\quad\mbox{ for }s\in [t_0+\tau_0,t]\}.
\]
Then we know that $T\ge t_0+2\tau_0$. 

We claim that $T=\infty$. If not, then by \eqref{eq:Lojasiewicz} and \eqref{eq:auxLojasiewicz}, it follows that for all $t\in [t_0+\tau_0,T]$, we have
\begin{align*}
-\frac{\ud}{\ud t} \left(J(v(\cdot,t))-J(\bar v)\right)^\theta= -\theta \left(J(v(\cdot,t))-J(\bar v)\right)^{\theta-1}\frac{\ud}{\ud t} J(v(\cdot,t))\ge c_0\theta\|v_t(\cdot,t)\|_{L^2(\Sigma_\ell)}.
\end{align*}
By integrating the above inequality and using Minkowski's integral inequality, we have 
\begin{align*}
\int_{t_0+\tau_0}^t\|v_s(\cdot,s)\|_{L^2(\Sigma_\ell)}\,\ud s \ge  \left\| \int_{t_0+\tau_0}^t v_s(\cdot,s)\,\ud s \right\|_{L^2(\Sigma_\ell)}= \| v(\cdot,t)-v(\cdot,t_0+\tau_0)\,\ud s \|_{L^2(\Sigma_\ell)}.
\end{align*}
By triangle's inequality, we have that
\begin{align*}
\|v(\cdot,t)-\bar v\|_{L^2(\Sigma_\ell)}&\le\|v(\cdot,t_0+\tau_0)-\bar v\|_{L^2(\Sigma_\ell)}+ \| v(\cdot,t)-v(\cdot,t_0+\tau_0)\,\ud s \|_{L^2(\Sigma_\ell)}\\
&\le\|v(\cdot,t_0+\tau_0)-\bar v\|_{L^2(\Sigma_\ell)}+\frac{1}{c_0\theta}\left(J(v(\cdot,t_0+\tau_0))-J(\bar v)\right)^\theta
\end{align*}
for all $t\in [t_0+\tau_0,T]$. Again, since $\|v(\cdot,t_0)-\bar v\|_{L^2(\Sigma_\ell)}<\delta$, we have 
\[
\|v(\cdot,t_0+\tau_0)-\bar v\|_{C^{2,\alpha}(\Sigma_\ell)}<\delta_3\left(\left(\frac{c_0\theta\delta_2(\va)}{2}\right)^{\frac{1}{\theta}}\right),
\]
and consequently,
\[
|J(v(\cdot,t_0+\tau_0))-J(\bar v)|<\left(\frac{c_0\theta\delta_2(\va)}{2}\right)^{\frac{1}{\theta}}.
\]
Hence, for all $t\in [t_0+\tau_0,T]$, we have
\[
\|v(\cdot,t)-\bar v\|_{L^2(\Sigma_\ell)}\le\|v(\cdot,t_0+\tau_0)-\bar v\|_{L^2(\Sigma_\ell)}+\frac{1}{2}\delta_2(\va)\le \delta_2(\va).
\]
This implies that 
\[
\|v(\cdot,t)-\bar v\|_{C^{2,\alpha}(\Sigma_\ell)}<\va\quad\mbox{ for }t\in [t_0+\tau_0,T+\tau_0].
\]
This contradicts the maximality of $T$. Therefore, $T=\infty$, and thus,
\[
\|v(\cdot,t)-\bar v\|_{C^{2,\alpha}(\Sigma_\ell)}<\va\quad\mbox{ for }t\ge t_0+\tau_0.
\]
This proves that $v(\cdot,t)\to\bar v$ in $C^{2,\alpha}(\Sigma_\ell)$ as $t\to\infty$.

Moreover, for all $t$ large, we have
\begin{align*}
\frac{\ud}{\ud t} \left(J(v(\cdot,t))-J(\bar v)\right)&= -\int_{\Sigma_\ell} (v^p)_t v_t\,\ud vol_g\\
&\le -C \int_{\Sigma_\ell} (v^p)_t (v^p)_t\,\ud vol_g\\
&\le-C \|\nabla J (v(\cdot,t)) \|^2_{L^2(\Sigma_\ell)}\\
&\leq -C \abs{J(v(\cdot,t))-J(\bar v)}^{2-2\theta},
\end{align*}
where we used Proposition \ref{prop:harnack-sub2} in the first inequality, the equation \eqref{eq:fde1critical} in the second inequality, and \eqref{eq:Lojasiewicz} in the last inequality. Since $0<\theta<1/2$, we have
\[
\frac{\ud}{\ud t} \left(J(v(\cdot,t))-J(\bar v)\right)^{2\theta-1}\ge (1-2\theta)C>0,
\]
and thus, 
\[
J(v(\cdot,t))-J(\bar v) \le C t^{\frac{1}{2\theta-1}}
\]
for all $t$ large. Then for $T$ large, we have
\begin{align*}
\left(\int_T^{2T}\Big(\int_{\Sigma_\ell} |v_t|^2\,\ud vol_g\Big)^{1/2}\,\ud t  \right)^{2}
& \le T\int_T^{2T}\int_{\Sigma_\ell} |v_t|^2\,\ud vol_g\,\ud t  \\
&\le C T \int_T^{2T}\int_{\Sigma_\ell} v^{p-1}|v_t|^2\,\ud vol_g\,\ud t\\
&=CT (J(v(\cdot,T))-J(v(\cdot,2T)))\\
&\le CT (J(v(\cdot,T))-J(\bar v))\\
&\le CT^\frac{2\theta}{2\theta-1}.
\end{align*}
Hence
\[
\int_T^{2T}\Big(\int_{\Sigma_\ell} |v_t|^2\,\ud vol_g\Big)^{1/2}\,\ud t \le C T^\frac{\theta}{2\theta-1},
\]
and thus,
\begin{align*}
\int_T^{\infty}\Big(\int_{\Sigma_\ell} |v_t|^2\,\ud vol_g\Big)^{1/2}\,\ud t&=\sum_{k=0}^\infty\int_{2^kT}^{2^{k+1}T}\Big(\int_{\Sigma_\ell} |v_t|^2\,\ud vol_g\Big)^{1/2}\,\ud t \\
&\le C T^\frac{\theta}{2\theta-1}\sum_{k=0}^\infty 2^\frac{\theta k}{2\theta-1}\\
&\le C T^\frac{\theta}{2\theta-1}.
\end{align*}
Now for $\tilde t>t\gg 1$, we have
\begin{align*}
\|v(\cdot,\tilde t)-v(\cdot,t)\|_{L^2(\Sigma_\ell)}= \left\|\int_{t}^{\tilde t} v_s(\cdot,s) \,\ud s\right\|_{L^2(\Sigma_\ell)}&\le \int_{t}^{\tilde t} \|v_s(\cdot,s)\|_{L^2(\Sigma_\ell)} \,\ud s\\
&\le \int_{t}^{\infty} \|v_s(\cdot,s)\|_{L^2(\Sigma_\ell)} \,\ud s\\
&\le Ct^\frac{\theta}{2\theta-1}.
\end{align*}
Therefore, by sending $\tilde t\to\infty$,  we obtain 
\[
\|v(\cdot,t)-\bar v\|_{L^2(\Sigma_\ell)}\le Ct^{-\gamma},
\]
where $\gamma=\frac{\theta}{1-2\theta}>0$. Since $v$ is bounded in all $C^k$ norms, by interpolation inequalities, we have
\[
\|v(\cdot,t)-\bar v\|_{C^2(\Sigma_\ell)}\le Ct^{-\gamma}.
\]
This finishes the proof of the decay rate.
\end{proof}
\begin{proof}[Proof of Theorem \ref{thm:criticalglobal}:]
By the change of variables \eqref{eq:cylinercoordinate}, the convergence of 
$u$ is reduced to the convergence of $w$. Under the assumptions of Theorem \ref{thm:criticalglobal}, $w$ satisfies \eqref{eq:fde1critical0}, and thus, Theorem \ref{thm:criticalglobal2} applies. Scaling back to $u$, Theorem \ref{thm:criticalglobal} follows.
\end{proof}

\begin{proof}[Proof of Theorem \ref{thm:subcriticalglobal20}:]
By the change of variables \eqref{eq:cylinercoordinate}, the convergence of 
$u$ is reduced to the convergence of $w$. Under the assumptions of Theorem \ref{thm:subcriticalglobal20}, $w$ satisfies \eqref{eq:fde1subcritical0}, and thus, Theorem \ref{thm:subcriticalglobal22} (which is a result of del Pino-S\'aez \cite{del}) applies. Scaling back to $u$, Theorem \ref{thm:subcriticalglobal20} follows.
\end{proof}

\section{Subcritical cases}\label{sec:subcriticalnoboundary}

In this section, we will prove Theorems \ref{thm:subcriticalglobal2} and \ref{thm:subcriticalglobal3}. We will prove Theorem \ref{thm:subcriticalglobal3} first, and Theorem \ref{thm:subcriticalglobal2} would follow.

\subsection{Uniform bounds}

We assume that $b$ is a smooth function, $1<p<\frac{n+2}{n-2}$ if $n>3$, and $1<p<\infty$ if $n=1,2$.

\begin{prop} \label{prop:harnack-sub-upperbound} 
Suppose $v$ is a positive smooth solution of \eqref{eq:fdemanifold1} satisfying \eqref{eq:vLpbound}.   Then there exists a constant $C$ depending only on $M,g,n,T^*,b,p$, and the $C_0$ in \eqref{eq:vLpbound} such that 
\[
v(x,t)\le C \quad \mbox{for all }x\in M,\ t>1. 
\]
\end{prop}
\begin{proof}
We only do the $n\ge 3$ case, since the other one is similar. We will use Moser's iteration and adapt that in Bonforte-V\'azquez \cite{BV} for the fast diffusion equations. 

Let $0<T_2<T_1<T_0$ be such that $|T_1-T_2|\le 1$, $\eta(t)$ be a smooth cut-off function so that $\eta(t)=0$ for all $t<T_2$, $0\le \eta(t)\le 1$ for $t\in [T_2,T_1]$, $\eta(t)=1$ for all $t>T_1$, and $|\eta'(t)|\le \frac{2}{T_1-T_2}$. Denote
\[
Q_1=M\times [T_1,T_0],\quad  Q_2=M\times [T_2,T_0].
\]
For $m>0$, define
\[
v_m=\min(v,m).
\]
Let $\beta\ge 0$. In the following, $C$ will be denoted as various constants that may change from lines to lines, but it will be \emph{independent} of $\beta, m,T_0,T_1,T_2$. 

We multiple $\eta^2v_m^\beta v$ on both sides of the equation \eqref{eq:fdemanifold1} and integrate over $Q_2$. Then by integration by parts, we have that
\[
\iint_{Q_2}\eta^2v_m^\beta v \frac{\partial}{\partial t} v^p+\iint_{Q_2}\eta^2\nabla_g v\nabla_g (vv_m^\beta)    \le  \frac{p}{T^*(p-1)} \iint_{Q_2} \eta^2v_m^\beta v^{p+1}+\iint_{Q_2} b \eta^2v_m^\beta v^2   .
\]
For the first term on the left-hand side, we have
\begin{align*}
\iint_{Q_2}\eta^2v_m^\beta v \frac{\partial}{\partial t} v^p
&=\frac{p}{p+1} \iint_{Q_2}\eta^2v_m^\beta  \frac{\partial}{\partial t} v^{p+1}\\
&=\frac{p}{p+1} \iint_{Q_2}\eta^2 \frac{\partial}{\partial t} (v_m^\beta v^{p+1})-  \eta^2 v^{p+1}\frac{\partial}{\partial t} v_m^\beta\\
&=\frac{p}{p+1} \iint_{Q_2}\eta^2 \frac{\partial}{\partial t} (v_m^\beta v^{p+1})-  \eta^2 v_m^{p+1}\frac{\partial}{\partial t} v_m^\beta\\
&=\frac{p}{p+1} \iint_{Q_2}\eta^2 \frac{\partial}{\partial t} (v_m^\beta v^{p+1})-  \frac{\beta}{p+1+\beta}\eta^2 \frac{\partial}{\partial t} v_m^{p+1+\beta}\\
&=\frac{p}{p+1} \int_{\Sigma_\ell}( v_m^\beta v^{p+1})(T)-  \frac{\beta}{p+1+\beta}(v_m^{p+1+\beta})(T)\\
&\quad+\frac{p}{p+1} \iint_{Q_2}2\eta\eta_t (v_m^\beta v^{p+1})-  \frac{\beta}{p+1+\beta}2\eta\eta_t  v_m^{p+1+\beta}\\
&\ge \frac{p}{p+1+\beta} \int_{M}( v_m^\beta v^{p+1})(T)-\frac{C}{T_1-T_2} \iint_{Q_2} v_m^\beta v^{p+1}.
\end{align*}
For the second term on the left-hand side, we have
\begin{align*}
\iint_{Q_2}\eta^2\nabla_g v\nabla_g (vv_m^\beta)&=\iint_{Q_2}\eta^2(\nabla_g v\nabla_g v v_m^\beta+\beta \nabla_g v\nabla_g v_m v v_m^{\beta-1})\\
&\ge \frac{1}{2(1+\beta)}\iint_{Q_2} \eta^2|\nabla_g (vv_m^{\frac{\beta}{2}})|^2.
\end{align*}
Combine all the estimates together, we have
\begin{equation}\label{eq:caccipoli1}
\int_{M}( v_m^\beta v^{p+1})(T)+\iint_{Q_1} |\nabla_g (vv_m^{\frac{\beta}{2}})|^2 \le \frac{C(1+\beta)}{T_1-T_2} \iint_{Q_2} v_m^\beta (v^{p+1}+v^2).
\end{equation}
Choose $s_0\in [T_1,T]$ such that
\[
\int_{M}( v_m^\beta v^{p+1})(s_0)\ge \frac{1}{2}\sup_{t\in [T_1,T]}\int_{\Sigma_\ell}( v_m^\beta v^{p+1})(t).
\]
Then we can replace $T$ by $s_0$ in \eqref{eq:caccipoli1}, and obtain that
\begin{equation}\label{eq:caccipoli2}
\sup_{t\in [T_1,T]}\int_{M}( v_m^\beta v^{p+1})(t)+\iint_{Q_1} |\nabla_g (vv_m^{\frac{\beta}{2}})|^2 \le \frac{C(1+\beta)}{T_1-T_2} \iint_{Q_2} v_m^\beta (v^{p+1}+v^2).
\end{equation}
In particular, 
\be\label{eq:interm}
\sup_{t\in [T_1,T]}\int_{M}( v_m^\beta v^{p+1})(t)\le \frac{C(1+\beta)}{T_1-T_2} \iint_{Q_2} v_m^\beta (v^{p+1}+v^2).
\ee

Using \eqref{eq:vLpbound} and H\"older's inequality, we have
\[
\frac{1}{C} \le \int_{M}v^{p+1}\le \left(\int_{M}v^{p+1+\beta}\right)^{\frac{p+1}{p+1+\beta}}vol(M)^{\frac{\beta}{p+1+\beta}}\le C\left(\int_{M}v^{p+1+\beta}\right)^{\frac{p+1}{p+1+\beta}}.
\]
Hence, 
\[
\left(\int_{M}v^{p+1+\beta}\right)^{\frac{1}{p+1+\beta}}\ge \frac{1}{C},
\]
where $C$ is independent of $\beta$. Then, since $p>1$, we have
\be\label{eq:holderaux}
\int_{M}v^2v_m^{\beta}\le \int_{M}v^{\beta+2}\le \left(\int_{M}v^{p+1+\beta}\right)^{\frac{\beta+2}{p+1+\beta}}vol(M)^{\frac{p-1}{p+1+\beta}}\le C\int_{M}v^{p+1+\beta}.
\ee
By sending $m\to\infty$ in \eqref{eq:interm}, we have
\begin{equation}\label{eq:caccipoliaux}
\sup_{t\in [T_1,T]}\int_{M}(v^{p+1+\beta})(t)\le \frac{C(1+\beta)}{T_1-T_2} \iint_{Q_2}  v^{p+1+\beta}.
\end{equation}

On the other hand, using the Sobolev inequality, we have that for each $\sigma \in (1,\frac{n}{n-2})$, 
\begin{align*}
\int_{M} (vv_m^{\frac{\beta}{2}})^{2\sigma}&=\int_{M} (vv_m^{\frac{\beta}{2}})^{2} (v^2v_m^{\beta})^{\sigma-1}\nonumber\\
&\le \left(\int_{M} (vv_m^{\frac{\beta}{2}})^{\frac{2n}{n-2}}  \right)^{\frac{2}{n-2}}\left(\int_{M} (v^2v_m^{\beta})^{\frac{(\sigma-1)n}{2}}\right)^{\frac{2}{n}}\nonumber\\
&\le C \left(\int_{M} \Big(|\nabla_g (vv_m^{\frac{\beta}{2}})|^2+v^2v_m^{\beta} \Big)\right)\left(\int_{M} v^{\frac{(\beta+2)(\sigma-1)n}{2}}\right)^{\frac{2}{n}}\nonumber\\
&\le C \left(\int_{M} \Big(|\nabla_g (vv_m^{\frac{\beta}{2}})|^2+v^{p+1+\beta} \Big)\right)\left(\int_{M} v^{\frac{(\beta+2)(\sigma-1)n}{2}}\right)^{\frac{2}{n}},
\end{align*}
where we used \eqref{eq:holderaux} in the last inequality. Integrating from $T_1$ to $T$, we have
\begin{align}
\iint_{Q_1} (vv_m^{\frac{\beta}{2}})^{2\sigma}&\le C \left(\iint_{Q_1} \Big(|\nabla_g (vv_m^{\frac{\beta}{2}})|^2+v^{p+1+\beta} \Big)\right)\left(\sup_{t\in [T_1,T]}\int_{M} v^{\frac{(\beta+2)(\sigma-1)n}{2}}(t)\right)^{\frac{2}{n}}\nonumber\\
&\le  \frac{C(1+\beta)}{T_1-T_2}  \left(\iint_{Q_2} v^{p+1+\beta} \right)\left(\sup_{t\in [T_1,T]}\int_{M} v^{\frac{(\beta+2)(\sigma-1)n}{2}}(t)\right)^{\frac{2}{n}},\label{eq:caccipoli3}
\end{align}
where we used \eqref{eq:caccipoli2} and \eqref{eq:holderaux}. Choose 
\[
\frac{(\beta+2)(\sigma-1)n}{2}=p+1+\beta,\quad\mbox{that is}\quad \sigma=1+\frac{2(p+1+\beta)}{n(\beta+2)}.
\]
One can verify that (recall $\beta\ge 0$)
\[
1<\sigma<\frac{n}{n-2}\quad\mbox{because}\quad p<\frac{n+2}{n-2}.
\]
Hence, from \eqref{eq:caccipoli3} and \eqref{eq:caccipoliaux}, we obtain, by sending $m\to\infty$ in the end, 
\begin{equation}\label{eq:moser1}
\iint_{Q_1} v^{(\beta+2)\sigma}\le C\left(\frac{1+\beta}{T_1-T_2}\right)^{1+\frac{2}{n}}  \left(\iint_{Q_2} v^{p+1+\beta} \right)^{1+\frac{2}{n}}.
\end{equation}
For $k=0,1,2,3,\cdots$, we define $\beta_0=0$,
\begin{align*}
\sigma_k&= \sigma=1+\frac{2(p+1+\beta_k)}{n(\beta_k+2)},\\
\beta_{k+1}&= (\beta_k+2)\sigma_k-p-1=\left(1+\frac 2n\right)\beta_k+\frac{2(p+1)}{n}+1-p,\\
q_k&=\beta_{k}+p+1
\end{align*}
Then we have $q_0=p+1$,  
\begin{align*}
q_{k+1}&=\left(1+\frac 2n\right)q_k+1-p=\left(1+\frac 2n\right)^{k+1}\left(p+1-\frac{(p-1)n}{2}\right)+\frac{(p-1)n}{2}\\
\beta_{k+1}&=\left(1+\frac 2n\right)q_k+1-p=\left(\left(1+\frac 2n\right)^{k+1}-1\right)\left(p+1-\frac{(p-1)n}{2}\right).
\end{align*}
Since $p<\frac{n+2}{n-2}$, we have $p_{k}$ is strictly increasing, and $p_k\to+\infty$ as $k\to\infty$.

Choose $t_0=\frac {1}{4}$, $t_1=c_0$, and  $t_{k+1}-t_k=c_0k^{-4}$ where  $c_0=(2\sum_{k=1}^\infty k^{-4})^{-1}>\frac 38$. Then $\lim_{k\to\infty}t_k=\frac 12$. Denote $Q_k=\Sigma_\ell\times [t_k,1]$. Then using \eqref{eq:moser1}, we have
\[
\iint_{Q_{k+1}} v^{q_{k+1}}\le C\left((1+\beta_k)k^4\right)^{1+\frac{2}{n}}  \left(\iint_{Q_k} v^{q_k} \right)^{1+\frac{2}{n}}.
\]
Hence
\begin{align*}
\left(\iint_{Q_{k+1}} v^{q_{k+1}}\right)^\frac{1}{q_{k+1}}&\le
 C^{\frac{1}{q_{k+1}}}\left(1+\frac 2n\right)^{\frac{2k(n+2)}{nq_{k+1}}} 
 \left(\iint_{Q_k} v^{q_k} \right)^{\frac{1}{q_{k}}\frac{2+n}{n}\frac{q_{k}}{q_{k+1}}}\\
 &\le C^{\frac{1}{q_{k+1}}\sum_{j=0}^k\left(\frac{n+2}{n}\right)^j} B^{\frac{1}{q_{k+1}}\sum_{j=0}^k (k-j)\left(\frac{n+2}{n}\right)^j}  \left(\iint_{Q_0} v^{q_0} \right)^{\frac{1}{q_{k+1}}\left(\frac{2+n}{n}\right)^{k+1}},
 \end{align*}
 where $B=\left(1+\frac 2n\right)^{\frac{2(n+2)}{n}}$. 
Since
 \begin{align*}
 \lim_{k\to\infty}\frac{1}{q_{k+1}}\sum_{j=0}^k\left(\frac{n+2}{n}\right)^j&=\frac{n}{2(p+1)-(p-1)n},\\
  \lim_{k\to\infty}\frac{1}{q_{k+1}}\sum_{j=0}^k(k-j)\left(\frac{n+2}{n}\right)^j&=\frac{n^2}{4(p+1)-2(p-1)n},\\
   \lim_{k\to\infty}\frac{1}{q_{k+1}}\left(\frac{n+2}{n}\right)^{k+1}&=\frac{2}{2(p+1)-(p-1)n},
 \end{align*}
by sending $k\to\infty$, we obtain
 \[
 \|v\|_{L^\infty(M\times [1/2,1])}\le C \left(\iint_{Q_0} v^{p+1} \right)^\frac{2}{2(p+1)-(p-1)n}\le C.
 \]
This finishes the proof, since the equation is translation invariant in the time variable.
\end{proof}

\begin{prop} \label{prop:low-bd} 
Let $v$ be as in Proposition \ref{prop:harnack-sub-upperbound}. Then there exists $C>0$ such that
\[
v(x,t)\ge \frac{1}{C} \quad \mbox{for }x\in M,\ t\ge 1.  
\]

\end{prop} 

\begin{proof}  
Let $L=-\Delta_g-b$, $Q= v^{-p}L v$, and $\alpha=\frac{p}{(p-1)T^*}$. Then we have 
\begin{align*}
\frac{\pa }{\pa t} Q&= -p v^{-p-1} \pa_t v Lv +v^{-p}L(\pa_t v)\nonumber\\
&=- v^{-2p} (-L v +\alpha v^{p}) Lv + \frac{1}{p}v^{-p} L(v^{-p+1}(- Lv+\alpha v^p))\nonumber\\
&=-\frac{1}{p}v^{-p} L(vQ)-\left(1-\frac{1}{p}\right) \alpha Q +Q^2\\
&=-\frac{1}{p}v^{-p}  \Big((Lv)Q -v \Delta_{g}Q -2\langle\nabla_{g} v, \nabla_{g} Q\rangle_g\Big) -\left(1-\frac{1}{p}\right) \alpha Q+Q^2  \\&
= \frac{1}{p}v^{-p+1}\Delta_{g}Q +\frac{2}{p}v^{-p}\langle\nabla_{g} v, \nabla_{g} Q\rangle_g + \left(1-\frac{1}{p}\right)  Q^2 -\left(1-\frac{1}{p}\right) \alpha Q \\& 
\ge   \frac{1}{p}v^{-p+1}\Delta_{g}Q +\frac{2}{p}v^{-p}\langle\nabla_{g} v, \nabla_{g} Q\rangle_g -\left(1-\frac{1}{p}\right) \alpha Q. 
\end{align*}
By the comparison principle, we have $Q\ge c_1:= \min\{\min_{M} Q(\cdot,0), 0\}$. Thus, $$Lv \ge c_1 v^{p},$$ 
that is
\[
-\Delta_g v - bv-(c_1v^{p-1}) v\ge 0\quad\mbox{on }M.
\]
By Proposition \ref{prop:harnack-sub-upperbound}, the term $c_1v^{p-1}$ is bounded. By the standard local Harnack inequality of linear elliptic equation with bounded coefficients (cf. Theorem 8.18 in Gilbarg-Trudinger \cite{GT2001}), Proposition \ref{prop:harnack-sub-upperbound} and \eqref{eq:vLpbound}, we have 
\[
\inf_{M} v(\cdot,t) \ge \frac{1}{C}\int_{M} v (\cdot,t)\,\ud vol_g \ge  \frac{1}{C} \int_{M} v^{p+1}(\cdot,t)\,\ud vol_g\ge \frac{1}{C},
\]
for some $C>0$ depending only on $M,g,n,T^*,b,p$, and the $C_0$ in \eqref{eq:vLpbound}.
\end{proof}

\subsection{Convergence}
\begin{proof}[Proof of Theorem \ref{thm:subcriticalglobal3}:]
It follows from Proposition \ref{prop:harnack-sub-upperbound} and Proposition \ref{prop:low-bd} that the solution $v$ is uniformly bounded from below away from zero and uniformly bounded from above on $M\times[1,+\infty)$. By the H\"older estimates and Schauder estimates of linear parabolic equations, and also the bootstrap arguments, $v$ is uniformly bounded in $C^{k}(M)$ for every $k>0$ on $[1,\infty)$. Then, the proofs of the convergence and the decay rate will be the same as those of  Theorem \ref{thm:criticalglobal2}.
\end{proof}

\begin{proof}[Proof of Theorem \ref{thm:subcriticalglobal2}]
It follows from the result of Bidaut-V\'eron-V\'eron \cite{BVeron} that the solutions to the stationary equation of \eqref{eq:fde1subcritical} under the assumptions of Theorem \ref{thm:subcriticalglobal2} have to be constants. Then Theorem \ref{thm:subcriticalglobal2} follows from Theorem \ref{thm:subcriticalglobal3} immediately.
\end{proof}
\begin{proof}[Proof of Theorem \ref{thm:subcriticalglobal}:]
It follows from Theorem \ref{thm:subcriticalglobal2}, in a similar way to the proof of Theorem \ref{thm:criticalglobal} or Theorem \ref{thm:subcriticalglobal20}.
\end{proof}

\bigskip

\noindent T. Jin

\noindent Department of Mathematics, The Hong Kong University of Science and Technology\\
Clear Water Bay, Kowloon, Hong Kong\\[1mm]
Email: \textsf{tianlingjin@ust.hk}

\medskip

\noindent J. Xiong

\noindent School of Mathematical Sciences, Laboratory of Mathematics and Complex Systems, MOE\\
Beijing Normal University, Beijing 100875, China\\[1mm]
 \textsf{Email: jx@bnu.edu.cn}

\end{document}